\documentclass[12pt,psamsfonts]{amsart}
\usepackage{fullpage,amssymb,amsfonts,amsmath}
\usepackage[all,cmtip]{xy}
\usepackage{enumerate}
\usepackage{mathrsfs}
\usepackage{comment}
\usepackage{hyperref} 
\usepackage[capitalise]{cleveref}
\usepackage{floatrow}
\usepackage{url}
\usepackage{todonotes}
	\let\svtodo\todo\renewcommand\todo[1]{\svtodo[inline]{#1}}
	
\let\oldtocsection=\tocsection

\let\oldtocsubsection=\tocsubsection

\let\oldtocsubsubsection=\tocsubsubsection

\renewcommand{\tocsection}[2]{\hspace{0em}\oldtocsection{#1}{#2}}
\renewcommand{\tocsubsection}[2]{\hspace{1em}\oldtocsubsection{#1}{#2}}
\renewcommand{\tocsubsubsection}[2]{\hspace{2em}\oldtocsubsubsection{#1}{#2}}
\theoremstyle{plain} 

\newtheorem{thm}{Theorem}[section]
\newtheorem{cor}[thm]{Corollary}
\newtheorem{prop}[thm]{Proposition}
\newtheorem{lem}[thm]{Lemma}

\newtheorem{quest*}{Question}
\newtheorem{prob*}{Problem}

\theoremstyle{definition}

\theoremstyle{remark}

\numberwithin{equation}{section}
\newcounter{notation}

\DeclareUrlCommand\DOI{}
\newcommand{\doi}[1]{\href{http://dx.doi.org/#1}{\DOI{doi:#1}}}



\crefname{figure}{Figure}{Figures}
\theoremstyle{plain}
\newtheorem*{thm*}{Theorem}
\crefname{thm}{Theorem}{Theorems}
\crefname{cor}{Corollary}{Corollarys}
\newtheorem*{cor*}{Corollary}
\crefname{cor*}{Corollary}{Corollarys}
\crefname{lem}{Lemma}{Lemmas}
\crefname{prop}{Proposition}{Propositions}
\crefname{conj}{Conjecture}{Conjectures}
\newtheorem*{conj*}{Conjecture}
\crefname{conj*}{Conjecture}{Conjectures}
\crefname{defn}{Definition}{Definitions}

\theoremstyle{remark}
\newtheorem*{rem*}{Remark}
\newtheorem*{rems*}{Remarks}

\def\addsymbol #1: #2#3{$#1$ \> \parbox{5.4in}{#2 \dotfill \pageref{#3}}\\} 
\def\addsymbolEND #1: #2#3{$#1$ \> \parbox{5.4in}{#2 \dotfill \pageref{#3}}}

\newcommand{\ds}{\displaystyle}

\newcommand{\kD}{\mathfrak{D}}
\newcommand{\kd}{\mathfrak{d}}

\renewcommand{\epsilon}{\varepsilon}

\newcommand{\sL}{\mathscr{L}}

\newcommand{\kn}{\mathfrak{n}}
\newcommand{\N}{\mathrm{N}}

\newcommand{\kN}{\mathfrak{N}}

\newcommand{\kp}{\mathfrak{p}}

\newcommand{\cO}{\mathcal{O}}

\newcommand{\Q}{\mathbb{Q}}

\newcommand{\R}{\mathbb{R}}

\renewcommand{\Re}{\mathrm{Re}}

\newcommand{\cS}{\mathcal{S}}

\newcommand{\sumP}{\sideset{}{'}\sum}

\title{The least unramified prime which does not split completely}

\author{Asif Zaman}
\address{
Asif Zaman \\
Department of Mathematics \\ 
University of Toronto   \\ 
40 St. George Street, Room 6290 \\
Toronto \\
Canada \\
M5S 2E4}
\email{asif@math.toronto.edu}

\date{April 10, 2017}

\begin{document}

\begin{abstract} Let $K/F$ be a finite extension of number fields of degree $n \geq 2$. We establish effective field-uniform unconditional upper bounds for the least norm of a prime ideal $\kp$ of $F$ which is degree 1 over $\Q$ and does not ramify or split completely in $K$. We improve upon the previous best known general estimates due to X. Li when $F = \Q$ and Murty--Patankar when $K/F$ is Galois. Our bounds are the first when $K/F$ is not assumed to be Galois and $F \neq \mathbb{Q}$.  \end{abstract}

\maketitle


\section{Introduction}

\subsection{History}
Let $K/F$ be a finite extension of number fields of degree $n \geq 2$. Define
\begin{align*}
\mathscr{P}(F) & = \{ \kp \text{ prime ideal of $F$ which is degree 1 over $\Q$} \},\\
\mathsf{P}(K/F) & = \min\{ \N^F_{\Q} \kp :  \kp \in \mathscr{P}(F) \text{ and $\kp$ does not ramify or split completely in $K$} \},\\
\mathsf{P}_*(K/F) & = \min\{ \N^F_{\Q} \kp :  \kp \in \mathscr{P}(F) \text{ and $\kp$ does not split completely in $K$} \}.
\end{align*}
The focus of this paper is to establish field-uniform upper bounds for $\mathsf{P}(K/F)$ and $\mathsf{P}_*(K/F)$.  The study of these quantities has classical origins and has been explored in a variety of cases. Indeed, when $K = \Q(\sqrt{d})$ is a quadratic field over $F = \Q$, this reduces to the problem of bounding the least quadratic nonresidue. Assuming the Generalized Riemann Hypothesis (GRH), Ankeny \cite{Ankeny} proved $\mathsf{P}( \Q(\sqrt{d})/\Q)  \ll (\log |d|)^{2}$.  Much less is known unconditionally and progress is notoriously difficult. Namely,
\begin{equation}
\mathsf{P}( \Q(\sqrt{d})/\Q) \ll_{\epsilon} |d|^{\frac{1}{4\sqrt{e}} + \epsilon}
\label{eqn:Burgess}
\end{equation}
for $\epsilon > 0$. Aside from the factor of $\epsilon$, this result of Burgess \cite{Burgess_57,Burgess_62} from over fifty years ago remains essentially the best known unconditional bound. 

More generally, when $K$ is Galois over $F$ of degree $n \geq 2$,  V.K. Murty \cite{VKM2} showed under the assumption of GRH for the Dedekind zeta function of $K$ that  
\begin{equation}
\mathsf{P}(K/F) \ll \Big(\frac{1}{n}\log D_K\Big)^2,
\label{eqn:GRH_Murty}
\end{equation}
where $D_K = |\mathrm{disc}(K/\Q)|$ and the implied constant is absolute. Murty remarks that the same analytic method can yield an unconditional estimate of the form $O_F(D_K^{1/2(n-1)})$. By a different approach involving geometry of numbers, Vaaler and Voloch \cite{Vaaler-Voloch} established an explicit variant of such an unconditional estimate for $\mathsf{P}_*(K/\Q)$ when $K$ is Galois over $\Q$.  

If $K$ is some finite extension of $\Q$ (not necessarily Galois) then, using an elegant argument, X. Li \cite{XLi} superseded this prior unconditional bound for $\mathsf{P}_*(K/\Q)$. Namely, he showed that
\begin{equation}
	\mathsf{P}_*(K/\Q) \ll_{\epsilon} D_K^{\frac{1+\epsilon}{4A(n-1)}}, 
	\label{eqn:XLi}
\end{equation}
where
\begin{equation}
A = A(n) = \sup_{\lambda > 0} \Big(\dfrac{1 - \frac{n}{n-1} e^{-\lambda}}{\lambda} \Big) \geq 1 - \sqrt{\frac{2}{n-1}}. 
\label{eqn:A_XLi}
\end{equation}
The key innovation of Li was to incorporate methods of Heath-Brown \cite{HBLinnik} for Dirichlet $L$-functions to obtain a stronger explicit inequality for the Dedekind zeta function. 

Recently, Murty and Patankar \cite[Theorem 4.1]{MurtyPatankar} adapted Li's argument to obtain the first unconditional field-uniform estimate for $\mathsf{P}(K/F)$ when $K$ is Galois over $F$. To introduce their result, let $N_F = 16$ if there is a sequence of fields $\Q = F_0 \subset F_1 \subset \cdots \subset F_m = F$ with each $F_{i+1}/F_i$ normal and $N_F = 4 [F:\Q]!$ otherwise. Define
\begin{equation}
B_F = \min\{ N_F \log D_F, c_1 D_F^{1/[F:\Q]} \},
\label{def:B_F}
\end{equation}
for some sufficiently small absolute constant $c_1 > 0$. Murty and Patankar proved if $K/F$ is Galois of degree $n \geq 2$ then
\begin{equation}
	\mathsf{P}(K/F) \leq C_F D_K^{\frac{4}{n-1}},
	\label{eqn:MurtyPatankar}
\end{equation}
where $C_F =  e^{O([F:\Q] (\log D_F)^2)} + e^{O(B_F)}$ and the implied constants are absolute.  Note the constant $C_F$ in the quoted result \eqref{eqn:MurtyPatankar} differs from \cite[Theorem 4.1]{MurtyPatankar} since there seems to be a typo stemming from equation (4.1) therein and its application in their proof. We remark that the dependence on $F$ in \eqref{eqn:MurtyPatankar} is natural given the current status of the effective Chebotarev Density Theorem \cite{LO} and the Brauer--Siegel theorem  \cite{Stark}. 

\subsection{Results} The primary focus of this paper is to improve the exponent of $D_K^{1/(n-1)}$ in both \eqref{eqn:XLi} and especially in \eqref{eqn:MurtyPatankar}. As a secondary objective, we consider both $\mathsf{P}_*(K/F)$ and $\mathsf{P}(K/F)$ for any finite extension $K/F$ which, in that generality, is new. We also demonstrate that one may take the non-split prime in \eqref{eqn:XLi} to be unramified in $K$ with some minor loss. 

Our approach is founded upon Li's argument blended with ideas of Heath-Brown \cite{HBLinnik} for zero-free regions of Dirichlet $L$-functions and their generalization in \cite{Zaman_2015a,Zaman_Thesis} for Hecke $L$-functions. Namely, we consider more general sums over prime ideals of $F$ which depend on a choice of polynomial. To state our main result, we introduce a definition: a polynomial $P(x) \in \R_{\geq 0}[x]$ is \emph{admissible} if $P(0) = 0, P'(0) = 1,$ and 
	\begin{equation}
	\Re\{ P(1/z) \} \geq 0 \qquad \text{for } \Re\{z\} \geq 1. 
	\label{def:AdmissiblePoly}
	\end{equation}

\begin{thm} \label{thm:1}
	Let $K/F$ be an extension of number fields of degree $n \geq 2$. Let $\epsilon > 0$ be fixed and $P(x) = \sum_{k=1}^d a_k x^d$ be a fixed admissible polynomial. There exists a prime ideal $\kp$ of $F$ such that $\kp$ does not split completely in $K$, $\kp$ is degree 1 over $\Q$, and
	\[
	  \N^F_{\Q} \kp \leq C_F D_K^{\frac{1+\epsilon}{4A(n-1)}},  
	\]
	where $C_F =  e^{O([F:\Q] (\log D_F)^2)} + e^{O(B_F)}$, $B_F$ is given by \eqref{def:B_F}, and 
	\begin{equation}
	A = A(n,P) = \sup_{\lambda > 0} \left( \Big[  P(1)   - \tfrac{n}{n-1} e^{-\lambda} \sum_{k=1}^d a_k \sum_{j=0}^{k-1} \frac{\lambda^j}{j!}\,  \Big] \Big/ \lambda  \right).
	\label{def:A}
	\end{equation}
	If $K/F$ is Galois then one may take $\kp$ to also be unramified in $K$.  All implied constants depend at most on $\epsilon$ and $P$. 
\end{thm}
\begin{rem*}
	~
	\begin{itemize}
		\item 	While $A$ depends on $n$, it is bounded above and below independent of $n$. In particular, if $P(x) = x + x^2$ then 
		\[
		A(n,P) \geq 1 - 2n^{-2/3},
		\]
		which improves over \eqref{eqn:A_XLi} as $n \rightarrow \infty$.  Moreover, the exponent $\tfrac{1+\epsilon}{4A}$ becomes a nearly sixteen-fold improvement over the exponent $4$ in \eqref{eqn:MurtyPatankar} as $n \rightarrow \infty$. With a different choice of $P$, we have by \cref{table:A_values} that $\frac{1+\epsilon}{4A} < \frac{5}{12}$ for all $n \geq 2$. This constitutes a nine-fold improvement over \eqref{eqn:MurtyPatankar} for all $n \geq 2$. 
		\item If $K/F$ is not assumed to be Galois, then one may still take $\kp$ to be unramified in $K$ but we show it satisfies the slightly weaker bound 
		\[
			\N^F_{\Q} \kp \ll  \big( C_F +  n^{\frac{3 P(1)}{A}} \big)   D_K^{ \frac{1+\epsilon}{4A(n-1)}}.
		\]
		By a classical result of Minkowski, recall that  $n \leq [K:\Q] \ll \log D_K$ so, unless $n$ is unusually large, this additional factor is negligible compared to $D_K^{\epsilon/(n-1)}$. 
	\end{itemize}

\end{rem*}

We restate \cref{thm:1} in the special case $F=\Q$. 

\begin{cor} \label{cor:2}
	Let $K$ be a number field of degree $n \geq 2$. Let $\epsilon > 0$ be fixed and $P(x) = \sum_{k=1}^d a_k x^d$ be a fixed admissible polynomial.  The least rational prime $p$ which does not split completely in $K$ satisfies
	\[
	p \ll D_K^{\frac{1+\epsilon}{4A(n-1)}}, 
	\]
	where $A = A(n,P)$ is given by \eqref{def:A}. If $K/\Q$ is Galois then one may also take $p$ to be unramified in $K$. Furthermore, if $P(x) = x + x^2$ then $A \geq 1 - 2n^{-2/3}$. All implied constants depend at most on $\epsilon$ and $P$. 
\end{cor}

Choosing a certain admissible polynomial $P(x) = P_{100}(x)$ of degree 100, say, \cref{cor:2} yields savings for every degree $n$ over the special case \eqref{eqn:XLi} where $P(x) = P_1(x) = x$. For example, if $K/\Q$ is an extension of  degree 5 then, by \cref{cor:2} with $P = P_{100}$,
\[
\mathsf{P}(K/\Q) \ll D_K^{1/8.7},
\]
whereas if $P = P_1$ then $1/8.7$ is replaced by $1/6.1$. See \cref{sec:AdmissiblePolys,table:A_values} for further details on these computations. 
 
 Finally, we describe the organization of the paper. \cref{sec:Counting} collects standard estimates related to counting prime ideals in a number field $F$. \cref{sec:Poly} contains an explicit inequality of the Dedekind zeta function and a generalization related to admissible polynomials. \cref{sec:proof} has the proof of \cref{thm:1} and \cref{sec:AdmissiblePolys} outlines the computation of admissible polynomials and \cref{table:A_values}.

 \subsection*{Notation} We henceforth adhere to the convention that all implied constants in all asymptotic inequalities $f\ll g$ or $f=O(g)$ are absolute with respect to all parameters and effectively computable.  If an implied constant depends on a parameter, such as $\epsilon$, then we use $\ll_{\epsilon}$ and $O_{\epsilon}$ to denote that the implied constant depends at most on $\epsilon$. 

\subsection*{Acknowledgements} The author would like to thank John Friedlander, Kumar Murty, and Jesse Thorner for their encouragement and helpful comments. 

\section{Counting prime ideals}
\label{sec:Counting}
Let $F$ be a number field of degree $n_F = [F:\Q]$ with discriminant $D_F = |\mathrm{disc}(F/\Q)|$ and ring of integers $\cO_F$. Denote $\N^F_{\Q}$ to be the absolute norm of $F$ over $\Q$. For each integral ideal $\kn \subseteq \cO_F$, define
\[
\Lambda_F(\kn) = \begin{cases} \log \N^F_{\Q}\kp & \text{if $\kn$ is a power of a prime ideal $\kp$,} \\ 0 & \text{otherwise.}	\end{cases}
\]

\begin{lem}
	\label{lem:PIT}
	Let $F$ be a number field and $\eta > 0$ be arbitrary. Define
	\begin{equation}
	X_0 = X_0(F, \eta) := \exp(10 n_F (\log D_F)^2)  + \exp(B_F \log(1/\eta)),
 	\label{def:X_0}
	\end{equation}
	where $B_F$ is defined by \eqref{def:B_F}. For $X \geq X_0$, 
	\[
	(1-\eta)X + O\Big( \frac{X}{(\log X)^2} \Big) \leq \sum_{\N^F_{\Q} \kn \leq X} \Lambda_F(\kn) \leq  (1 + \eta)X + O\Big( \frac{X}{(\log X)^2} \Big). 
	\]	
	All implied constants are absolute.
\end{lem}
\begin{proof}
	The effective Chebotarev Density Theorem	 \cite{LO}  implies that, for $X \geq X_0$
	\begin{equation}
	\Big| \sum_{\N^F_{\Q} \kn \leq X} \Lambda_F(\kn) - X \Big| \leq X^{\beta} + O(X \exp(- c  n_F^{-1/2} (\log X)^{1/2} ) ),
	\label{eqn:EffectiveCDT_original}
	\end{equation}
 where $c > 0$ is some absolute constant and $\beta > 1/2$ is a real zero of the Dedekind function of $K$, if it exists. By a theorem of Stark \cite[Theorem 1']{Stark}, any real zero $\beta$ of the Dedekind zeta function $\zeta_F(s)$ satisfies
\begin{equation*}
\beta < 1 - \frac{1}{B_F},
\label{eqn:Stark}
\end{equation*}
where $B_F$ is given by \eqref{def:B_F}. Hence, by 
\eqref{def:X_0}, we have $X^{\beta} = X \cdot X^{\beta-1} \leq \eta X$. By Minkowski's bound, observe that $n_F \ll \log D_F \ll \sqrt{\log X}$. It follows that $n_F^{-1/2} (\log X)^{1/2} \gg (\log X)^{1/4}$, so the error term in \eqref{eqn:EffectiveCDT_original}  is crudely bounded by $O(X /(\log X)^2)$. 
\end{proof}

\begin{lem} \label{lem:LogPowerPrimeSum}
Let $k \geq 1$ be an integer and $\eta \in (0,1/2)$ be arbitrary. Let $X \geq Y \geq X_0$ where $X_0 = X_0(F,\eta)$ is defined by \eqref{def:X_0}. Denote $E_{k-1}(t) = \sum_{j=0}^{k-1} t^j/j!$.  Then
\begin{equation*}
	\begin{aligned}
		& \sum_{Y < \N^F_{\Q}\kn \leq X} \dfrac{\Lambda_F(\kn)}{\N^F_{\Q}\kn^{\sigma}} (\log \N^F_{\Q}\kn)^{k-1} \\
		& \qquad\qquad 
			\geq  \dfrac{(k-1)!}{(\sigma-1)^{k}} \cdot ( 1 - \eta ) \big( Y^{1-\sigma}  - X^{1-\sigma} E_{k-1}( (\sigma-1) \log X) \big) + O_k\Big(\frac{1}{(\sigma-1)^{k-1}}\Big)
	\end{aligned}
\end{equation*}
uniformly for $1 < \sigma < 2$.
\end{lem}
\begin{proof}
	This is a combination of partial summation and \cref{lem:PIT}. We include the proof for sake of completeness. Define $\psi_F(t) = \sum_{\N^F_{\Q}\kn < t} \Lambda_F(\kn)$ for $t > 1$. By partial summation,
	\begin{equation*}
		\begin{aligned}
			\sum_{Y < \N^F_{\Q}\kn \leq X} \dfrac{\Lambda_F(\kn)}{\N^F_{\Q}\kn^{\sigma}} (\log \N^F_{\Q}\kn)^{k-1} = \psi_F(X) X^{-\sigma} (\log X)^{k-1} - \int_{Y}^X \psi_F(t) \dfrac{d}{dt}\big[t^{-\sigma} (\log t)^{k-1} \big] dt. 
		\end{aligned}
	\end{equation*}
	By \cref{lem:PIT}, it follows for $t \geq Y \geq X_0$ that
	\[
	-\psi_F(t) \dfrac{d}{dt}\big[t^{-\sigma} (\log t)^{k-1} \big] \geq (1- \eta)\sigma t^{-\sigma} (\log t)^{k-1} \big\{ 1 + O_k(\frac{1}{\log t}) \big\}. 
	\]
	Discarding the first term in the previous equation by positivity and using the above inequality, we deduce that
	\begin{equation}
	\sum_{Y < \N^F_{\Q}\kn \leq X} \dfrac{\Lambda_F(\kn)}{\N^F_{\Q}\kn^{\sigma}} (\log \N^F_{\Q}\kn)^{k-1}  
	\geq  (1- \eta) \int_{Y}^X  t^{-\sigma} (\log t)^{k-1} dt + O_k\Big( \int_Y^X t^{-\sigma} (\log t)^{k-2} dt \Big).
	\label{eqn:LogPowerIntegral} 
	\end{equation}
	The remaining integrals are computed by parts. One iteration yields:
	\begin{equation*}
		\begin{aligned}
			\int_{Y}^X  t^{-\sigma} (\log t)^{k-1} dt 
			&	= \dfrac{Y^{1-\sigma}}{\sigma-1} (\log Y)^{k-1} - \dfrac{X^{1-\sigma}}{\sigma-1}  (\log X)^{k-1} + \dfrac{k-1}{(\sigma-1)} \int_Y^X t^{-\sigma} (\log t)^{k-2} dt.
		\end{aligned}
	\end{equation*}
	Proceeding by induction, we conclude that
	\begin{equation*}
		\begin{aligned}
				\int_{Y}^X  t^{-\sigma} (\log t)^{k-1} dt 
				&	= (k-1)!  \sum_{j=0}^{k-1}  \Big[ \dfrac{Y^{1-\sigma} (\log Y)^{k-1-j} }{(k-1-j)!(\sigma-1)^{j+1}}  - \dfrac{ X^{1-\sigma} (\log X)^{k-1-j}}{(k-1-j)! (\sigma-1)^{j+1}}  \Big] \\
				&	= \dfrac{(k-1)!  }{(\sigma-1)^k} \big( Y^{1-\sigma} E_{k-1}( (\sigma-1)\log Y) - X^{1-\sigma} E_{k-1}( (\sigma-1) \log X) \big).\\
		\end{aligned}
	\end{equation*}
	Substituting this expression in \eqref{eqn:LogPowerIntegral} and observing $1 \leq E_{k-1}(t) \leq e^t$  (in order to simplify the main term and error term involving $Y$), we obtain the desired result.
\end{proof}


\begin{lem} \label{lem:RamifiedPrimes}
Let $K$ be a finite extension of $F$. Let $V(K/F)$ be the set of places $v$ of $F$ which ramify in $K$ and $\kp_v$ be the prime ideal of $F$ attached to $v$. Unconditionally,
	\[
	\sum_{\substack{ v \in V(K/F) }} \log \N^F_{\Q} \kp_v \leq \log D_K.
	\]
	If $K/F$ is Galois then
	\[
	\sum_{\substack{ v \in V(K/F) }} \frac{\log \N^F_{\Q} \kp_v}{\N^F_{\Q} \kp_v} \leq \sqrt{\frac{2 [F:\Q] }{[K:F]}\log D_K }.
	\]
\end{lem}
\begin{proof}
The unconditional inequality follows from the well-known formula
\[
\log D_K = [K:F] \log D_F + \log \N^F_{\Q} \kd_{K/F},
\]
where $\kd_{K/F} = \N^K_{F} \kD_{K/F}$ and $\kD_{K/F}$ is the relative different ideal of $K/F$. If $K/F$ is Galois then, by Cauchy-Schwarz and \cite[Proposition 5, Section I.3]{Ser1},
\begin{equation*}
	\begin{aligned}
		\sum_{\substack{ v \in V(K/F) }} \frac{\log \N^F_{\Q} \kp_v}{\N^F_{\Q} \kp_v} 
			& \leq \Big( \sum_{\substack{ v \in V(K/F) }} \log \N^F_{\Q} \kp_v \Big)^{1/2} \Big( \sum_{\substack{ v \in V(K/F) }} \frac{\log \N^F_{\Q} \kp_v}{\N^F_{\Q} \kp_v^2} \Big)^{1/2} \\
			& \leq \big( \frac{2}{[K:F]} \log D_K \big)^{1/2} \Big([F:\Q] \sum_{p} \frac{\log p}{p^2} \Big)^{1/2} \\
			& \leq \sqrt{ \frac{2 [F:\Q] \log D_K }{[K:F]}},
	\end{aligned}
\end{equation*} 
as desired. In the above, we used that there are at most $[F:\Q]$ prime ideals $\kp$ of $F$ above a given rational prime $p$ and $\sum_p \frac{\log p}{p^2} < 1$. 
\end{proof}


\section{Polynomial explicit inequality}
\label{sec:Poly}
Let $K$ be a number field with $D_K = |\mathrm{disc}(K/\Q)|$ and let $\zeta_K(s)$ be the Dedekind zeta function of $K$. Our starting point is a variant of the classical explicit formula.

\begin{prop}[Thorner--Z.]
\label{prop:Classical_Explicit_Inequality}
Let $K$ be a number field and $0 < \epsilon < 1/8$ be arbitrary. There exists $\delta = \delta(\epsilon) > 0$ such that
\[
-\Re\Big\{ \dfrac{\zeta_K'}{\zeta_K}(s) \Big\} \leq (\tfrac{1}{4} + \epsilon) \log D_K  + \Re\Big\{\dfrac{1}{s-1}\Big\} - \sum_{|1+it-\rho| < \delta} \Re\Big\{ \dfrac{1}{s-\rho} \Big\} + O_{\epsilon}([K:\Q]),
\]		
uniformly for $s = \sigma + it$ with $1 < \sigma < 1+\epsilon$ and $|t| \leq 1$.
\end{prop}
\begin{rem*}
	The value $1/4$ is derived from the convexity bound for $\zeta_K(s)$ in the critical strip.
\end{rem*}

\begin{proof}
	 This follows from \cite[Proposition 2.6]{ThornerZaman_2016a}; similar variants appear in \cite{XLi,KadiriNg}.  See  \cite[Proposition 3.2.3]{Zaman_Thesis} for details.
\end{proof}

We would like to analyze more general sums  over prime ideals by considering higher derivatives of the logarithmic derivative  $-\dfrac{\zeta_K'}{\zeta_K}(s)$. This generalization (\cref{prop:Polynomial_Explicit_Inequality}) is motivated by the work of Heath-Brown \cite[Section 4]{HBLinnik}. 

Given a polynomial $P(x) \in \R_{\geq 0}[x]$ of degree $d$ with $P(0) = 0$, write 
\[
P(x) = \sum_{k=1}^d a_k x^k
\]
and define
\begin{equation}
	\cS(\sigma) = \cS_K(\sigma; P) := \sum_{\kn \subseteq \cO_K} \frac{\Lambda_K(\kn)}{\N\kn^{\sigma}} \sum_{k=1}^d a_k \frac{( (\sigma-1) \log \N\kn)^{k-1}}{(k-1)!} 
	\label{def:S_PolyWeight}
\end{equation}
for $\sigma > 1$. Recall the definition of an admissible polynomial from \eqref{def:AdmissiblePoly}. Note the condition $P'(0) = 1$ is imposed for normalization purposes since it implies $a_1 = 1$.

\begin{prop} 
\label{prop:Polynomial_Explicit_Inequality}
Let $0 < \epsilon < 1/8$ and $\lambda > 0$ be arbitrary. If $P(x) = \ds\sum_{k=1}^d a_k x^k$ is an admissible polynomial of degree $d$ then 
\[
\cS(\sigma) = \cS_K(\sigma, P)  \leq  (\tfrac{1}{4}+\epsilon) \log D_K + \frac{P(1)}{\sigma-1}    + O_{\epsilon,P,\lambda}([K:\Q])
\]
uniformly for 
\[
1 < \sigma \leq 1 + \min\Big\{ \epsilon, \frac{\lambda [K:\Q]}{\log D_K} \Big\}.
\]
\end{prop}
\begin{proof} This is essentially \cite[Proposition 5.2]{Zaman_2015a} with \cref{prop:Classical_Explicit_Inequality} used in place of \cite[Lemma 4.3]{Zaman_2015a}. Our argument proceeds similarly but we exhibit a different range of $\sigma$ which is more suitable for our purposes. For simplicity, denote $\sL = \log D_K$ and $n_K = [K:\Q]$. Define
\[
P_2(x) := \sum_{k=2}^d a_k x^k = P(x) - a_1 x. 
\]
From the functional equation of $\zeta_K(s)$, it follows by \cite[Lemma 2.6]{Zaman_2015a} that
\begin{equation*}
	\begin{aligned}
		\dfrac{(-1)^{k-1}}{(k-1)!} \dfrac{d^{k-1}}{ds^{k-1}}\Big( -\dfrac{\zeta_K'}{\zeta_K}(s) \Big)
		& = \dfrac{1}{(s-1)^k}  - \sum_{\rho} \dfrac{1}{(s-\rho)^k}  + \dfrac{1}{\sigma^k} - \dfrac{(-1)^{k}}{(k-1)!} \dfrac{d^{k-1}}{ds^{k-1}}\Big(\dfrac{\gamma_K'}{\gamma_K}(s) \Big) \\
		& = \dfrac{1}{(s-1)^k}  - \sum_{\rho} \dfrac{1}{(s-\rho)^k}  + O(n_K) \\
	\end{aligned}
\end{equation*}
for $\Re\{s\} > 1$. On the other hand, from the Euler product of $\zeta_K(s)$ one can verify that
\[
(-1)^{k-1} \dfrac{d^{k-1}}{ds^{k-1}}\Big( -\dfrac{\zeta_K'}{\zeta_K}(s) \Big)
= \sum_{\kn \subseteq \cO_K} \frac{\Lambda_K(\kn)}{\N\kn^{s}} (\log \N\kn)^{k-1}
\]
for $\Re\{s\} > 1$. Comparing these two expressions at $s=\sigma$ with \eqref{def:S_PolyWeight} and taking real parts\footnote{This is redundant as the expression is already real, but clarifies the later use of admissibility of $P$.}, we deduce that
\begin{equation}
	\cS(\sigma; P_2) = \dfrac{1}{\sigma-1} \sum_{k=2}^d a_k \Re\Big\{ 1 - \sum_{\rho} \Big(\dfrac{\sigma-1}{\sigma-\rho}\Big)^{k-1} \Big\} + O_P(n_K)
	\label{eqn:Poly_2}
\end{equation}
for $\sigma > 1$. 
We wish to restrict the sum over zeros $\rho$ in \eqref{eqn:Poly_2} to $|1-\rho| < \delta$ for $\delta = \delta(\epsilon) > 0$ given by \cref{prop:Classical_Explicit_Inequality}. Observe by \cite[Lemma 5.4]{LO} that
\begin{equation}
\begin{aligned}
\sum_{\substack{\rho = \beta+i\gamma \\ |1-\rho| \geq \delta}} \Re\Big\{ \Big(\dfrac{\sigma-1}{\sigma-\rho}\Big)^{k-1} \Big\} 
&	\ll_{\epsilon,k} (\sigma-1)^{k-1} \sum_{T=0}^{\infty} \sum_{\substack{\rho = \beta+i\gamma \\ T \leq |\gamma| \leq T+1} } \dfrac{1}{1+t^2} \\
&	\ll_{\epsilon,k,\lambda} \Big(\dfrac{n_K}{\sL}\Big)^{k-1} \sum_{T=0}^{\infty} \dfrac{\sL + n_K \log(T+3)}{1+T^2}\\
&	\ll_{\epsilon,k,\lambda} \Big(\dfrac{n_K}{\sL}\Big)^{k-1} \sL \\
& \ll_{\epsilon,k,\lambda} n_K,	
\end{aligned}
\label{eqn:Poly_2_ExtraZeros}
\end{equation}
since $k \geq 2, \sigma < 1 + \frac{\lambda n_K}{\sL},$ and $n_K \ll \sL$. Now, consider the linear polynomial $P_1(x) = a_1 x = x$ as $P'(0) = 1$. By \cref{prop:Classical_Explicit_Inequality}, we find that
\[
\cS(\sigma; P_1) \leq (\tfrac{1}{4}+\epsilon) \sL  + a_1\Re\Big\{ \frac{1}{\sigma-1} - \sum_{|1-\rho| < \delta} \frac{1}{\sigma-\rho} \Big\} + O_{\epsilon}(n_K).
\]
Notice $\cS(\sigma; P) = \cS(\sigma; P_1) + \cS(\sigma; P_2)$ by linearity in the second argument. Hence, we may combine the above with \eqref{eqn:Poly_2} and \eqref{eqn:Poly_2_ExtraZeros} yielding
\begin{equation*}
	\begin{aligned}
\cS(\sigma; P) 
& \leq (\tfrac{1}{4}+\epsilon) \sL + \dfrac{1}{\sigma-1} \sum_{k=1}^d a_k \Re\Big\{ 1 - \sum_{|1-\rho| < \delta} \Big(\dfrac{\sigma-1}{\sigma-\rho}\Big)^{k-1} \Big\} + O_{\epsilon,P,\lambda}(n_K) \\
& \leq   (\tfrac{1}{4}+\epsilon) \sL+ \dfrac{1}{\sigma-1}  P(1) - \dfrac{1}{\sigma-1} \sum_{|1-\rho| < \delta} \Re\big\{ P\big(\tfrac{\sigma-1}{\sigma-\rho}\big) \big\} + O_{\epsilon,P,\lambda}(n_K) \\
& \leq  (\tfrac{1}{4}+\epsilon) \sL + \dfrac{P(1)}{\sigma-1}  + O_{\epsilon,P,\lambda}(n_K). \\
	\end{aligned}
\end{equation*}
In the last step, we noted $\Re\big\{P\big(\tfrac{\sigma-1}{\sigma-\rho}\big)\big\} \geq 0$ by admissibility of $P$. 
\end{proof}

\section{Proof of \cref{thm:1}}
\label{sec:proof}
We will deduce \cref{thm:1} from the following result.

\begin{thm} \label{thm:MainTheorem} 
Let $K/F$ be an extension of number fields of degree $n \geq 2$ and $X \geq Y$. Assume one of the following holds: 
\begin{enumerate}
	\item[(A1)] Every prime ideal $\kp$ of $F$ which is degree 1 over $\Q$ with $Y < \N^F_{\Q}\kp \leq X$ splits completely in $K$. 
	\item[(A2)] Every unramified prime ideal $\kp$ of $F$ which is degree 1 over $\Q$ with $Y < \N^F_{\Q}\kp \leq X$ splits completely in $K$. 
	\item[(A3)] Assumption (A2) holds and $K/F$ is Galois. 
\end{enumerate}
Let $0 < \epsilon < \tfrac{1}{8}$ be arbitrary and $P(x) = \sum_{k=1}^d a_k x^d$ be an admissible polynomial. For $M = M(\epsilon, P)$ sufficiently large, define
\begin{equation}
Y_0  = \begin{cases} 
	X_0 & \text{if (A1) or (A3) hold,} \\ 
	X_0 + Mn & \text{if (A2) holds},  		
 \end{cases}
 \label{eqn:Y_0}
\end{equation}
where $X_0 = X_0(F, \eta)$ is given by \eqref{def:X_0} and $\eta = \eta(\epsilon,P)$ is sufficiently small. If $X \geq Y \geq Y_0$ then
\begin{equation}
(1- \epsilon) A \log X \leq \frac{(\tfrac{1}{4}+\epsilon) \log D_K}{(n-1)} + \frac{n}{n-1} P(1) \log Y + O_{\epsilon,P}([F:\Q]),
\label{eqn:MainTheorem}
\end{equation}
where $A = A(n,P)$ is given by \eqref{def:A}. 
\end{thm}

\subsection{Proof of \cref{thm:1} from \cref{thm:MainTheorem}} Without loss, assume $\epsilon \in (0, \tfrac{1}{8})$.  Taking $Y = Y_0$ and rescaling $\epsilon > 0$ appropriately in \cref{thm:MainTheorem} yields
\[
A \log X \leq \frac{(\tfrac{1}{4}+\epsilon) \log D_K}{(n-1)} + 3 P(1)  \log Y_0 + O_{\epsilon,P}([F:\Q]).
\] 
By considering cases arising from \eqref{eqn:Y_0} and fixing $\epsilon$ and $P$, this yields the desired bound for $X$ in all cases. Moreover, if $P(x) = x+x^2$ and $\lambda > 0$ then
\begin{equation*}
\begin{aligned}
	A(n,P) \geq \dfrac{2 - \frac{n}{n-1} e^{-\lambda}(2 + \lambda)}{\lambda}  
		& \geq -\frac{2}{(n-1)\lambda}  + \frac{n}{n-1}- \frac{n \lambda^2}{6(n-1)} \\
		& = \frac{n}{n-1}\Big( 1 - \frac{2}{n\lambda} - \frac{\lambda^2}{6}\Big) \\
		& = \frac{n}{n-1} \Big(1 - \frac{6^{2/3}}{2n^{2/3}} \Big) \geq 1 - \frac{2}{n^{2/3}},
\end{aligned}	
\end{equation*}
upon setting $\lambda = \sqrt[3]{6/n}$. \hfill \qed
\subsection{Proof of \cref{thm:MainTheorem}}
Let $0 < \lambda < \lambda(\epsilon,P)$ where $\lambda(\epsilon,P)$ is some sufficiently large constant and let $\sigma = 1 + \frac{\lambda}{\log X}$. One can verify $A = A(n,P) \geq A(2,P) \gg_P 1$ and $A \ll P(1)$ from \eqref{def:A}. Thus, by \eqref{eqn:MainTheorem}, we may assume without loss that $X \geq e^{\lambda(\epsilon,P)/\epsilon}$ and $X \geq D_K^{1/4(n-1)}$. This implies that $1 < \sigma < 1+ \min\{ \epsilon, \frac{4 \lambda(\epsilon,P) [K:\Q]}{\log D_K} \}$. Now, letting $\kD_{K/F}$ be the relative different of $K/F$, consider
\begin{equation}
\begin{aligned}
S & := \sum_{\substack{Y < \N^K_{\Q} \kN \leq X \\ (\kN,\kD_{K/F}) = 1}} \frac{\Lambda_K(\kN)}{\N^K_{\Q}\kN^{\sigma}} \sum_{k=1}^d a_k \frac{( (\sigma-1) \log \N^K_{\Q}\kN)^{k-1}}{(k-1)!}.
\end{aligned}
\end{equation}
By the positivity of the terms and \cref{prop:Polynomial_Explicit_Inequality}, it follows that
\begin{equation}
S \leq \cS(\sigma; P) \leq \dfrac{P(1)}{\sigma-1} +  (\tfrac{1}{4}+ \tfrac{\epsilon}{2}) \log D_K + O_{\epsilon,P}([K:\Q]).
\label{eqn:LNSP_upper_bound}
\end{equation}
On the other hand, by any of (A1), (A2), or (A3), each unramified prime ideal of $F$ splits completely into $[K:F]$ prime ideals. Hence, denoting $\kd_{K/F} = \N^K_F\kD_{K/F}$, we have that
\begin{equation}
\label{eqn:SplitAssumption}
S
 =  [K:F] \sum_{\substack{ Y < \N^F_{\Q}\kn \leq X \\ (\kn, \kd_{K/F}) = 1}} \dfrac{\Lambda_{F}(\kn)}{\N^F_{\Q}\kn^{\sigma}} \sum_{k=1}^d a_k \frac{( (\sigma-1) \log \N^F_{\Q}\kn)^{k-1}}{(k-1)!} \geq [K:F] \sum_{k=1}^d a_k  (S_k - R_k - T_k),
\end{equation}
where 
\begin{equation*}
\begin{aligned}
S_k & = \dfrac{(\sigma-1)^{k-1}}{(k-1)!}  \sum_{Y < \N^F_{\Q}\kn \leq X} \dfrac{\Lambda_F(\kn)}{\N^F_{\Q}\kn^{\sigma}} ( \log \N^F_{\Q}\kn)^{k-1}, \\
R_k & = \dfrac{(\sigma-1)^{k-1}}{(k-1)!}  \sum_{\substack{Y < \N^F_{\Q}\kn \leq X \\ (\kn, \kd_{K/F}) \neq 1}} \dfrac{\Lambda_F(\kn)}{\N^F_{\Q}\kn^{\sigma}} ( \log \N^F_{\Q}\kn)^{k-1}, \\
T_k & = \dfrac{(\sigma-1)^{k-1}}{(k-1)!}  \sumP_{Y < \N^F_{\Q}\kn \leq X } \dfrac{\Lambda_F(\kn)}{\N^F_{\Q}\kn^{\sigma}} ( \log \N^F_{\Q}\kn)^{k-1}. \\
\end{aligned}	
\end{equation*}
Here $\sum'$ indicates a restriction to ideals $\kn = \kp^j$ where $\kp$ is of degree $\geq 2$ over $\Q$ and $j \geq 1$. We estimate each $S_k$ using \cref{lem:LogPowerPrimeSum} with $\eta = \eta(\epsilon,P)$ sufficiently small to deduce that
\[
\sum_{k=1}^d a_k S_k 
 \geq \dfrac{1- \eta}{\sigma-1} \sum_{k=1}^d a_k (Y^{1-\sigma}  -X^{1-\sigma} E_{k-1}((\sigma-1)\log X) )  + O_P(1). 
\]
Since $X \geq Y$, $\sigma = 1 + \frac{\lambda}{\log X}$, and $e^{-t} \geq 1 - t$ for $t > 0$, we have that $Y^{1-\sigma} \geq 1 - (\sigma-1) \log Y$. The above equation therefore implies that
\begin{equation}
\frac{1}{1-\eta}  \sum_{k=1}^d a_k S_k  \geq  \Big( \frac{ P(1) - e^{-\lambda} \sum_{k=1}^d a_k E_{k-1}(\lambda)}{\lambda}  \Big) \log X  - P(1) \log Y + O_{P}(1)  
\label{eqn:SplitPrime_Contribution}
\end{equation}
To estimate $R_k$, we claim that 
\begin{equation}
\sum_{k=1}^d a_k R_k \leq \frac{\epsilon}{2[K:F]} \log D_K + O_{\epsilon,P}([F:\Q]). 
\label{eqn:RamifiedPrimes_Contribution}
\end{equation}
We divide into cases according to assumptions (A1), (A2), and (A3).  
\begin{itemize}
	\item If (A1) holds then $R_k = 0$ for all $k$ which trivially yields the claim. 
	\item If (A2) holds then, as $\lambda < \lambda(\epsilon,P)$ and $\sigma = 1 + \frac{\lambda}{\log X}$,
	\[
	\sum_{k=1}^d a_kR_k \ll_{\epsilon,P} \sum_{\substack{Y < \N^F_{\Q}\kn \leq X \\ (\kn, \kd_{K/F}) \neq 1}} \dfrac{\Lambda_F(\kn)}{\N^F_{\Q}\kn^{\sigma}} \ll_{\epsilon,P} \sum_{\substack{\N^F_{\Q}\kp > Y \\ \kp \mid \kd_{K/F}}} \dfrac{\log \N^F_{\Q}\kp}{\N^F_{\Q}\kp}.
	\]
	Since $Y \geq Y_0 \geq M[K:F]$ from \eqref{eqn:Y_0} and $M = M(\epsilon,P)$ is sufficiently large, it follows by \cref{lem:RamifiedPrimes} that
	\[
	\sum_{k=1}^d a_kR_k \leq \frac{\epsilon}{2[K:F]} \log D_K. 
	\]
	\item If (A3) holds then we argue as above and apply \cref{lem:RamifiedPrimes} in the $K/F$ Galois case to deduce that
	\[
	\sum_{k=1}^d a_k R_k \ll_{\epsilon,P} \sqrt{\frac{2[F:\Q]}{[K:F]} \log D_K}.
	\]
	By AM-GM, claim  \eqref{eqn:RamifiedPrimes_Contribution} follows. 
\end{itemize}
This proves the claim in all cases. Finally, to estimate $T_k$, we similarly observe that
\[
\sum_{k=1}^d a_k T_k \ll_{\epsilon,P} \sumP_{Y < \N^F_{\Q}\kp \leq X} \dfrac{\log \N^F_{\Q} \kp}{\N^F_{\Q}\kp^{\sigma}} \ll_{\epsilon,P} [F:\Q] \sum_{p} \dfrac{\log p}{p^{2\sigma}} \ll_{\epsilon,P} [F:\Q]. 
\]
Combining \eqref{eqn:LNSP_upper_bound}, \eqref{eqn:SplitAssumption}, \eqref{eqn:SplitPrime_Contribution},  \eqref{eqn:RamifiedPrimes_Contribution}, and the above, it follows that
\begin{equation}
(n-1) a(\lambda) \log X -   \eta n  b(\lambda) \log X \leq (\tfrac{1}{4}+\epsilon) \log D_K + n P(1) \log Y + O_{\epsilon,P}([K:\Q]),
\label{eqn:penultimate}
\end{equation}
where $n = [K:F]$,
\begin{equation*}
\begin{aligned}
a(\lambda) = a(\lambda; n, P) & =   \Big( \frac{ P(1) - \frac{n}{n-1}e^{-\lambda} \sum_{k=1}^d a_k E_{k-1}(\lambda)}{\lambda}  \Big), \\
b(\lambda) = b(\lambda; P) & = \Big( \frac{ P(1) - e^{-\lambda} \sum_{k=1}^d a_k E_{k-1}(\lambda)}{\lambda}  \Big). 
\end{aligned}
\end{equation*}
One can verify that the supremum of $b(\lambda)$ over $\lambda > 0$ exists and $A = A(n,P) = \displaystyle \sup_{\lambda > 0} a(\lambda)$ is bounded independent of $n$.  By taking $\eta = \eta(\epsilon,P)$ sufficiently small, we may therefore assume that  $\eta b(\lambda) < \tfrac{n-1}{n} \epsilon A$. Hence, \eqref{eqn:penultimate} implies
\[
(n-1) [ a(\lambda) - \epsilon A] \log X \leq (\tfrac{1}{4}+\epsilon) \log D_K + n P(1) \log Y + O_{\epsilon,P}([K:\Q]).
\]
Dividing both sides by $n-1$ and taking the supremum over $0 < \lambda < \lambda(\epsilon,P)$ yields the desired result, except for the range of $\lambda$ in definition of $A$. By straightforward calculus arguments, the supremum of $a(\lambda)$ occurs at $\lambda = \lambda_{n,P} > 0$ and one can verify that $\lambda_{n,P}$ is bounded above independent of $n$. Hence, for $\lambda(\epsilon, P)$ sufficiently large,
\[
\sup_{0 < \lambda < \lambda(\epsilon,P)} a(\lambda) = \sup_{\lambda > 0} a(\lambda) = A. 
\] 
This completes the proof.   \hfill \qed 

\section{Admissible polynomials with large values}
\label{sec:AdmissiblePolys}
Here we outline the computation of admissible polynomials $P(x)$ such that $P(1)$ is large which leads to large values for $A(n,P)$ in \cref{thm:1}. The key lemma for our calculations follows from arguments in \cite[Section 4]{HBLinnik} based on the maximum modulus principle.

\begin{lem}[Heath-Brown]
	\label{lem:AdmissiblePoly}
	A polynomial $P(x) \in \R_{\geq 0}[x]$ satisfying $P(0) = 0$ and $P'(0) = 1$ is admissible provided
	\[
	\Re\Big\{P\Big( \frac{1}{1+iy} \Big)\Big\} \geq 0 \qquad \text{for all $y \geq 0$.}
	\]
\end{lem}

For each integer $d \geq 1$, write $P(x) = \sum_{k=1}^d a_k x^k$ where $a_k \geq 0$ and $a_1 = 1$. We wish to  determine $a_2,\dots,a_d$ such that $P(1) = 1 + a_2 + \dots + a_d$ is maximum. From \cref{lem:AdmissiblePoly}, it suffices to verify that for all $y \geq 0$, 
\[
\sum_{k=1}^d a_k\dfrac{ \Re\{ (1-iy)^k \} }{(1+y^2)^k} \geq 0,
\text{ or equivalently, }
\sum_{k=1}^d a_k (1+y^2)^{d-k} \sum_{0 \leq j \leq k/2} (-1)^j {k \choose 2j} y^{2j} \geq 0.
\]
Expanding the above as a polynomial in $y$, let $\mathbf{a} = (a_1,a_2,a_3,\dots,a_d)$ and $C_{2j}^{(d)} = C_{2j}^{(d)}(\mathbf{a})$ be the coefficient of $y^{2j}$ for $0 \leq j \leq d-1$; all other coefficients are zero. As $a_1 = 1$, one can see that $C_0^{(d)} = 1 + a_2 + \dots + a_d = P(1)$. Therefore, $P(x)$ is admissible if the remaining $d-1$ coefficients $C_{2j}^{(d)}$ for $1 \leq j \leq d-1$ are non-negative. Notice $C_{2j}^{(d)}$ are linear expressions in $a_2,\dots,a_d$. Thus, one may apply the simplex method to maximize the objective function $P(1) = 1 + a_2 + \dots + a_d$ given the system of linear inequalities $\{ C_{2j}^{(d)}(\mathbf{a}) \geq 0\}_{j=1}^{d-1} \cup \{ a_j \geq 0\}_{j=2}^{d-1}$. Based on computational evidence for $1 \leq d \leq 100$, the maximum of this linear system occurs precisely when $C_{2j}^{(d)}(\mathbf{a}) = 0$ for all $1 \leq j \leq d-1$. We suspect this scenario is always the case, but we did not seriously investigate it as that is not our aim. 

\begin{table}
\begin{tabular}{l|c|c|c|c} 
$n$ & $4 A(n, P_{100}) \geq $ & $\lambda = \lambda(n, P_{100}) $ & $4A(n, P_1) \geq$ & $\lambda = \lambda(n, P_1)$ \\
\hline
 2 &  2.444 &  21.68 &  1.493 &  1.678 \\ 
 3 &  2.734 &  17.63 &  1.827 &  1.189 \\ 
 4 &  2.904 &  15.50 &  2.039 &  .9613 \\ 
 5 &  3.021 &  14.11 &  2.193 &  .8244 \\ 
 6 &  3.108 &  13.10 &  2.311 &  .7310 \\ 
 7 &  3.176 &  12.33 &  2.406 &  .6624 \\ 
 8 &  3.231 &  11.70 &  2.485 &  .6094 \\ 
 9 &  3.277 &  11.19 &  2.553 &  .5669 \\ 
 10 &  3.316 &  10.75 &  2.611 &  .5318 \\ 
 20 &  3.530 &  8.340 &  2.951 &  .3554 \\ 
 50 &  3.720 &  6.043 &  3.293 &  .2147 \\ 
 100 &  3.814 &  4.763 &  3.483 &  .1486 \\ 
 200 &  3.878 &  3.764 &  3.625 &  .1035 \\ 
 500 &  3.931 &  2.764 &  3.757 &  .0646 \\ 
 1000 &  3.956 &  2.191 &  3.826 &  .0454 \\ 
 2000 &  3.971 &  1.737 &  3.876 &  .0319 \\ 
 5000 &  3.984 &  1.279 &  3.921 &  .0201 \\ 
 10000 &  3.990 &  1.015 &  3.944 &  .0142 \\ 
\end{tabular}
\caption{Values of $A = A(n,P_d)$ when $d=100$ versus $d=1$.}
\label{table:A_values}
\end{table}
\begin{figure}
	\includegraphics[scale=0.65]{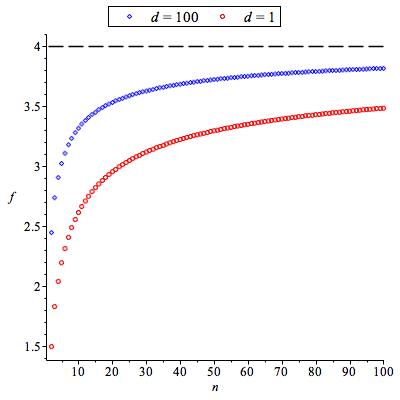}
	\caption{Plot of $f(n) = 4A(n,P_d)$ for $2 \leq n \leq 100$ with $d=1$ (\textcolor{red}{red circles}) below and $d=100$ (\textcolor{blue}{blue diamonds}) above.}
	\label{fig:A_values}
\end{figure}

	Thus, for each integer $d \geq 1$, let $P_d(x)$ be the polynomial associated to the unique solution $\mathbf{a}$ (if it exists) satisfying $C_{2j}^{(d)}(\mathbf{a}) =0$ for $1 \leq j \leq d-1$. For example,  
\[
P_1(x) = x, \quad P_2(x) = x +x^2, \quad P_3(x) = x + x^2 + \tfrac{2}{3} x^3, \quad P_4(x) = x + x^2 + \tfrac{4}{5}x^3 + \tfrac{2}{5}x^4.
\]
These are the same polynomials exhibited in \cite[Section 4]{HBLinnik}. 
Estimate \eqref{eqn:XLi} is based on the choice of $P_1(x) = x$. Setting $d = 100$, we may compute $P_{100}(x)$ and subsequently $A(n,P_{100})$ in \cref{table:A_values} for fixed values of $n$. One can compare $A(n,P_{100})$ with $A(n,P_1)$ in \cref{table:A_values,fig:A_values} above to observe the savings afforded by \cref{cor:2} over \eqref{eqn:XLi}.

\bibliographystyle{alpha}
\bibliography{LNSP}

\end{document}